\documentclass[a4paper,10pt,reqno]{amsart}
\usepackage{amsmath, enumerate}
\usepackage{amsthm}
\usepackage{graphicx}
\usepackage{amssymb}
\usepackage[utf8x]{inputenc}
\usepackage{fancyhdr}
\usepackage{calc}
\usepackage{url}
\usepackage[margin=2cm]{geometry}
\usepackage[english]{babel}
\usepackage[T1]{fontenc}
\usepackage[usenames,dvipsnames]{color}
\usepackage{esint}

\newcommand{\R}{\mathbb{R}}
\newcommand{\Sph}{\mathbb{S}^{d-1}_+}
\newcommand{\norm}[1]{\left\lVert#1\right\rVert}
\def\eps{\varepsilon}
\def\super{\overline}

\newtheorem{theorem}{Theorem}
\newtheorem{prop}{Proposition}
\newtheorem{lemma}{Lemma}
\newtheorem{remark}{Remark}
\newtheorem{cor}{Corollary}
\newtheorem{definition}{Definition}

\begin{document}

\title{Liouville-type theorems for the Lane-Emden equation in the half-space and cones}
\author{Louis Dupaigne, Alberto Farina and Troy Petitt}

\address{Louis Dupaigne: Institut Camille Jordan, UMR CNRS 5208, Universit\'{e} Claude Bernard Lyon 1, 43 boulevard du 11 novembre 1918, 69622 Villeurbanne cedex (France)}
\email{dupaigne@math.univ-lyon1.fr}
\address{Alberto Farina: LAMFA, UMR CNRS 7352, Universit\'{e} Picardie Jules Verne, 33 rue St. Leu, 80039 Amiens (France)}
\email{alberto.farina@u-picardie.fr}
\address{Troy Petitt: Dipartimento di Matematica, Politecnico di Milano, Piazza Leonardo da Vinci 32, 20133 Milano (Italy)}
\email{troy.petitt@polimi.it}

\maketitle

\section{Introduction}
Let $d\geq1$ and $p>1$. We consider 
the Lane-Emden equation 
\begin{equation}\label{LEQ}
\left\{
    \begin{aligned}
		-\Delta u &\,=\, \vert u\vert^{p-1}u & \textrm{in}\;\; \R^d_+\,, \\
		u &\,=\, 0 & \textrm{on}\;\; \partial\R^d_+\,,
	\end{aligned}
\right.
\end{equation}
posed in the upper half-space $\R^d_+:=\{x=(x',x_d)\colon\,x'\in\R^{d-1},x_d\in(0,\infty)\}$ with homogeneous Dirichlet boundary conditions. It is conjectured that the above equation has only one nonnegative solution, $u=0$.  

Gidas and Spruck in \cite{GS} showed that this is indeed the case if $1<p\le p_S(d)$, where
$$
p_S(d)=\frac{d+2}{(d-2)_+}
$$
is Sobolev's critical exponent. In the case $p> p_S(d)$, only partial results are available: Dancer \cite{D} considered bounded nonnegative solutions and proved in this case that $u=0$ if $p<p_S(d-1)$. The second named author \cite{Farina} improved Dancer's result showing that the only bounded nonnegative solution is $u=0$ if $p<p_{JL}(d-1)$, where $p_{JL}$ is the Joseph-Lundgren stability exponent given by 
\begin{equation}\label{JL exponent}
p_{JL}(d):=\frac{(d-2)^2-4d+8\sqrt{d-1}}{(d-2)(d-10)_+}.
\end{equation}
Finally Chen, Lin, and Zou \cite{CLZ} proved that no bounded nonnegative solution $u\neq0$ of \eqref{LEQ} exists for any $p>1$. In all these results, the fact that any nonnegative bounded solution of \eqref{LEQ} is monotone in the $x_d$-direction, i.e. $\partial u/\partial{x_d}>0$ in $\R^d_+$, is crucially used. In fact, Sirakov, Souplet and the first named author proved in \cite{DSS} that, more generally, no nontrivial monotone solution of \eqref{LEQ} exists, whether bounded or not.
Note that monotone solutions are stable, meaning they verify additionally that
\begin{equation}\label{def stable}
    p\int\,\vert u\vert^{p-1}\varphi^2\;dx\leq\int|\nabla\varphi|^2\;dx,
\end{equation}
for all $\varphi\in C^1_c(\R^d_+)$. Even more generally, consider solutions which are stable only outside a compact set $K\subset\R^d_+$, i.e. such that \eqref{def stable} holds for $\varphi\in C^1_c(\R^d_+\setminus K)$. The second named author proved in \cite[Theorem 9(b)]{Farina} that there is no such solution except $u=0$, provided $1<p<p_{JL}(d)$. We improve this result as follows.
\begin{theorem}\label{To prove:SOAC}
Let $p>1$ and let $u\in C^2(\R^d_+)\cap C(\overline{\R^d_+})$ be a solution of \eqref{LEQ} stable outside a compact set. Then $u=0$.
\end{theorem}
The above theorem can be partly extended to the following class of weak solutions.
\begin{definition}\label{weaksolution} Let $H$ denote the space of measurable functions $u$ defined on the upper half-space $\R^d_+$ such that $u\in H^1(\R^d_+\cap B_R)\cap L^{p+1}({\R^d_+\cap B_R})$ for every $R>0$ and $u=0$ on $\partial\R^d_+$ in the sense of traces. Then, $u$ is a weak solution of \eqref{LEQ} if $u\in H$ and it satisfies the equation $-\Delta u = \vert u\vert^{p-1}u$ in the sense of distributions. 
\end{definition}
Then, 
\begin{cor}\label{cor:SOAC}
Let $p>1$ and let $u$ be a nonnegative weak stable solution of \eqref{LEQ}. Then $u=0$.
\end{cor}

\begin{remark}
Corollary \ref{cor:SOAC} is sharp in the following sense. By Theorem 1.1 in \cite{BVPV}, for $p\in(\frac{d+1}{d-1},p_S(d-1))$, there exists a {\rm singular} solution of the form $u(x)=r^{-\frac2{p-1}}v(\theta)$, where $r=\vert x\vert$, $\theta=x/r$ and $v\in C^2(\overline {\mathbb S^{d-1}_+}) \cap H^1_0(\mathbb S^{d-1}_+)$ is positive and radial (w.r.t the geodesic distance to the north pole). Observe that $u\in H$ if only if $d \geq 3$ and $p\in(p_S(d),p_S(d-1))$, so that the equation is satisfied in the weak sense. However, $u$ is always unstable, see Lemma \ref{u to v} and Lemma \ref{Liouville for v} below. So, the stability assumption cannot be completely removed from the statement of Corollary \ref{cor:SOAC}.
\end{remark}

Next, we extend our study of the Lane-Emden equation to cones, i.e. we consider the equation
\begin{equation} \label{main}
\left\{
\begin{aligned}
-\Delta u &= \vert u\vert^{p-1}u&\quad\text{in $\Omega$,}\\
u&=0&\quad\text{on $\partial\Omega$,}
\end{aligned}
\right. 
\end{equation}
where 
\begin{equation}\label{cone}
\Omega=\left\{
r \theta \, : \, r\in(0,+\infty), \theta\in A
\right\},
\end{equation}
$d\ge 2$ and $A\subset \mathbb{S}^{d-1}$ is a subset of the unit sphere of dimension $d-1$. 

Busca proved that there are no positive solutions $u\in C^2(\Omega)\cap C(\super\Omega)$ of \eqref{main} in the case where $\Omega$ is a convex cone strictly contained in $\R^d_+$, see \cite{busca}. For such a cone, positive solutions are monotone (hence stable\footnote{i.e., they satisfy \eqref{def stable} for all $\varphi \in C^1_c(\Omega)$.} in $\Omega$) and $A=\Omega\cap\mathbb S^{d-1}$ is geodesically convex (see Lemma \ref{geometric0}), hence star-shaped with respect to the north pole (up to a suitable rotation, see Lemma \ref{geometric2}). So the following theorem extends both Busca's result and Theorem \ref{To prove:SOAC}.

\begin{theorem}\label{To prove:SOAC2}
Let $A$ be a $C^{2,\alpha}$ domain contained in the (open) upper-half sphere $\mathbb{S}^{d-1}_+$, which is star-shaped with respect to the north pole. Let $\Omega$ be given by \eqref{cone}. If $p>1$ and $u\in C^2(\Omega)\cap C(\overline{\Omega})$ is a solution of \eqref{main} stable outside a compact set \footnote{i.e., $u$ satisfies \eqref{def stable} for all $\varphi \in C^1_c(\Omega \setminus K)$, where $K \subset \Omega$ is a fixed compact set.}, then $u=0$.
\end{theorem}

\begin{remark}
We recall that a set $A\subseteq\mathbb{S}^{d-1}$ is said to be star-shaped with respect to a point $\theta_0\in A$ if for every $\theta\in A$ any minimal geodesic path from $\theta_0$ to $\theta$ remains inside $A$. Note that when $A\subseteq\mathbb{S}^{d-1}_+$, there is at most one such minimal geodesic path. 
\end{remark}

Finally, for general cones (not necessarily star-shaped, not necessarily contained in a half-space), as follows from the proof of Theorem \ref{To prove:SOAC2}, we have the following partial result
\begin{cor}\label{lastcor}
Let $\Omega\subset\R^d$ be a cone given by \eqref{cone} and $p>1$. Assume that
$$
p\neq p_S(d)\quad\text{and}\quad    p\mu-\frac{(d-2)^2}{4}+(p-1)\lambda_1\geq0,
$$
where $\mu=\frac{2}{p-1}\left(d-1-\frac{p+1}{p-1}\right)$ and $\lambda_1$ is the principal eigenvalue of the Laplace-Beltrami operator on $A=\Omega\cap\mathbb S^{d-1}$.
If $u\in C^2(\Omega)\cap C(\overline{\Omega})$ is a solution of \eqref{main} stable outside a compact set, then $u=0$.
\end{cor}

\begin{remark}
In particular, the Corollary applies to all $p\ge p_0$ for some (explicitly computable) $p_0$ depending on $d$ and $\lambda_1$. 
\end{remark}

\begin{remark}
Since $\Omega$ is not smooth at its vertex, some care is needed when dealing with regularity/integrability properties of the solution. The reader interested only in the proof of Theorem \ref{To prove:SOAC} can safely skip these considerations in the proof presented below. 
\end{remark}





\section{Proof of Theorem \ref{To prove:SOAC2}}

We begin by proving that any classical solution of \eqref{main} is a weak solution of \eqref{main} in the following sense.
\begin{definition}\label{weak definition 2}
Set $\Omega_R := \Omega \cap B_R$ for every $R>0$. We say that a measurable function $u$ defined on $\Omega$ is a weak solution of \eqref{main} if $u\in H^1(\Omega_R)\cap L^{p+1}(\Omega_R)$ for every $R>0$, $u$ solves $ - \Delta u \,=\, \vert u\vert^{p-1}u $ in the sense of distributions of $\Omega$ and if for all $\psi\in C^1_c(\R^d)$ we have $u \psi \in H^1_0(\Omega)$.
\end{definition}

\begin{lemma}\label{ls} Assume $ p>1$ and let $u\in C^2(\Omega)\cap C(\super\Omega)$ be a solution of \eqref{main}. Then, $u$ is a weak solution of \eqref{main}. 
Furthermore, every weak solution $u$ to \eqref{main} satisfies, 
\begin{equation}\label{eqfaibleplus}
      \int_{\Omega}\nabla u \nabla \phi =  \int_{\Omega} 
      \vert u \vert^{p-1} u \phi, 
\end{equation}
for all $\phi \in H^1_0(\Omega)\cap L^{p+1}(\Omega)$ with bounded support. 


\end{lemma}

%



\proof Multiply equation \eqref{main} by $u\zeta^2$, $\zeta\in C^{0,1}_{c}(\R^d \setminus \{0\}),$ and integrate by parts to get 
\begin{align*}
\int \vert \nabla u\vert^2\zeta^2 &= \int \vert u\vert^{p+1}\zeta^2 - 2\int u\zeta\nabla u\cdot\nabla\zeta\\
&\le \int \vert u\vert^{p+1}\zeta^2 + \frac12 \int \vert \nabla u\vert^2\zeta^2 + 2 \int u^2 \vert\nabla\zeta\vert^2\\
\end{align*}
Hence,
\begin{equation}\label{quasiH1}
\int \vert \nabla u\vert^2\zeta^2\le 2\| u \|_{L^\infty({\Omega_R})}^{p+1} \int \zeta^2 + 4
\| u \|_{L^\infty({\Omega_R})}^2\int  \vert\nabla\zeta\vert^2.
\end{equation}
Choosing $\zeta=\psi\zeta_n$, where $\zeta_n$ is a standard cut-off function away from the origin, and passing to the limit in the latter inequality we see that \eqref{quasiH1} remains true for every $\zeta=\psi\in C^1_c(\R^d)$. Hence, $u\in H^1(\Omega_R) \cap L^{p+1} (\Omega_R)$ for every $R>0$ and $u$ is a solution to $ - \Delta u \,=\, \vert u\vert^{p-1}u $ in the sense of distributions of $\Omega$. 
Also, $u \psi \in H^1_0(\Omega_R)$ for any $ R>>1$, since $\psi$ has compact support,\footnote{Recall that: $v \in H^1(A) \cap C^0(\super A)$ and $v = 0 $ on $\partial A$ implies $v \in H^1_0(A)$ for every open set $A \subset \R^d$.} and thus $u \psi \in H^1_0(\Omega)$.
Finally, \eqref{eqfaibleplus} follows by a standard density argument (by passing to the limit in the distributional formulation of $ - \Delta u \,=\, \vert u\vert^{p-1}u $).

\hfill\qed

\begin{remark}\label{remark7}
If $\Omega=\R^d_+$, Definitions \ref{weaksolution} and \ref{weak definition 2} coincide.
\end{remark}

Next we recall the following 

\begin{definition}\label{def-stability-for cones}
A weak solution of \eqref{main} is 
\begin{itemize}
    \item \texttt{stable} if $u$ satisfies \eqref{def stable} for all $\varphi \in C^1_c(\Omega)$,
    \item \texttt{stable outside a compact set} if there is a compact set $K \subset \Omega$ such that $u$ satisfies \eqref{def stable} for all $\varphi \in C^1_c(\Omega \setminus K)$. 
\end{itemize}
\end{definition}

Let us focus next on 
\subsection{The subcritical and critical cases $p\le p_{S}(d)$} 

\

This part of the proof is inspired by \cite{esteban-lions} but some care is needed since the domain $\Omega$ is not necessarily smooth (at the origin)\footnote{Take e.g. $A=\mathbb S^{d-1}\cap\{x\in\R^d : \vert x_d\vert<1/2\}$ and observe in passing that $\Omega$ is not contained in a half-space in this case.}  and so solutions might not belong to 
$W^{2,q}(\Omega_R)$ with $q\in(1,+\infty)$ as in their case. The following integral estimate holds.
\begin{lemma}\label{l2bis}
Let $1<p\le p_S(d)$. Let $u\in C^2(\Omega)\cap C(\super\Omega)$ be a solution of \eqref{main} stable outside a compact set $K$. Then, $u\in L^{p+1}(\Omega)$ and $\vert\nabla u\vert\in L^2(\Omega)$.
\end{lemma}

\begin{proof}
Let $u\in C^2(\Omega)\cap C(\super\Omega)$ be a solution of \eqref{main} stable outside a compact set $K$ and let $\varphi \in C^1_c(\R^d)$ supported on $\R^d \setminus K$. Then, $\phi = u \varphi^2$ has bounded support and belongs to $H^1_0(\Omega)\cap L^{p+1}(\Omega)$ (thanks to Lemma \ref{ls} and Definition \ref{weak definition 2}) and so, from \eqref{eqfaibleplus}, we get 

\begin{equation}\label{weak sol}
    \int
    \nabla u\nabla(u\varphi^2)=\int
    \vert u\vert^{p+1}\varphi^2.
\end{equation}
We estimate the left-hand side of \eqref{weak sol} using the stability of $u$ in $\Omega \setminus K$ 
\begin{align}\label{weak sol2}
    \int
    \vert u\vert^{p+1}\varphi^2=
    \int
    \nabla u\nabla(u\varphi^2)&=\int
    |\nabla(u\varphi)|^2-\int
    u^2|\nabla\varphi|^2 \\
    &\geq p\int
    \vert u\vert^{p+1}\varphi^2-\int\,u^2|\nabla\varphi|^2,\nonumber 
\end{align}
where in the latter we have used that $\phi = u \varphi^2$ belongs to $H^1_0(\Omega \setminus K)$, and so it can be inserted into \eqref{def stable}. 

Combining the above lines and observing that $p>1$, we reach
\begin{equation}\label{stable result}
    (p-1)\int
    \vert u\vert^{p+1}\varphi^2\leq\int
    u^2|\nabla\varphi|^2.
\end{equation}
Returning to \eqref{weak sol2}, it follows that
\begin{equation}\label{stable result2}
    \frac {p-1}p\int
    \vert \nabla (u\varphi)\vert^2\leq\int
    u^2|\nabla\varphi|^2.
\end{equation}
Taking $\varphi:=\psi^m$, for some some smooth non-negative function $\psi \in C^1_c(\R^d)$ supported on $\R^d \setminus K$ and $m>1$ gives
\begin{equation*}
    (p-1)\int
    \vert u\vert^{p+1}\psi^{2m}\leq\,m^2\int
    u^2\psi^{2m-2}|\nabla\psi|^2.
\end{equation*}
Using H\"older's inequality,
\begin{equation*}
    (p-1)\int
    \vert u\vert^{p+1}\psi^{2m}\leq\,m^2\left(\int
    \vert u\vert^{p+1}\psi^{(m-1)(p+1)}\right)^{\frac{2}{p+1}}\left(\int
    |\nabla\psi|^{2\frac{p+1}{p-1}}\right)^{\frac{p-1}{p+1}}.
\end{equation*}
We can take $m:=\frac{p+1}{p-1}$, so the above inequality simplifies to 
\begin{equation}\label{Lp+1 control tf}
    \int
    \,\vert u\vert^{p+1}\psi^{2m}\leq\,\left[\frac{(p+1)^2}{(p-1)^3}\right]^{\frac{p+1}{p-1}}\int
    |\nabla\psi|^{2\frac{p+1}{p-1}}.
\end{equation}
Fix $R_0>0$ such that $K\subset\Omega_{R_0}$ and let $R>2R_0$ so $u$ is stable outside of $\Omega_R$. If $\psi_R(x)=\psi_1(x/R)$ is a standard cutoff function, then $\psi=(1-\psi_{R_0})\psi_R$ is supported in $B_{2R}\setminus K$ and so we may apply \eqref{Lp+1 control tf}. It follows that 
\begin{equation}\label{bound0}
    \int_{\Omega_R\setminus\Omega_{2R_0}}\,\vert u\vert^{p+1}\le
    C_p\int_{\Omega_{2R_0}}
    |\nabla\psi|^{2\frac{p+1}{p-1}}+
    C_p\int_{\Omega_{2R}\setminus \Omega_{2R_0}}
    |\nabla\psi|^{2\frac{p+1}{p-1}}
    \leq\,C_{p,d}(R_0^{d-2\frac{p+1}{p-1}}+R^{d-2\frac{p+1}{p-1}}).
\end{equation}
Using the same test function in \eqref{stable result2} yields similarly
\begin{equation}\label{bound1}
    \int_{\Omega_R\setminus\Omega_{2R_0}}\,\vert\nabla u\vert^{2}
    \leq\,C_{p,d}(R_0^{d-2\frac{p+1}{p-1}}+R^{d-2\frac{p+1}{p-1}}).
\end{equation}
The lemma follows by letting $R\to+\infty$, since $p\le p_S(d)$ implies $d-2\frac{p+1}{p-1}\le 0$.
\end{proof}

Thanks to Lemma \ref{l2bis}, we derive the following Poho\v{z}aev identity.
\begin{prop}\label{prop1} Under the assumptions of Lemma \ref{l2bis}, there holds
$$
\left(1-\frac d2\right)\int_\Omega\vert\nabla u\vert^2 + \frac d{p+1}\int_\Omega \vert u\vert^{p+1} = 0.
$$
\end{prop}
\begin{proof}
Let $(\zeta_n)$ denote a standard truncation sequence near the origin and $(\theta_n)$ a standard truncation sequence near infinity. Set $\psi_n=\zeta_n\theta_n$ for $n\ge1$. Multiply equation \eqref{main} by $(x\cdot\nabla u)\psi_n$ and integrate by parts. Then,
$$
-\int_\Omega\Delta u (x\cdot\nabla u)\psi_n=\int_\Omega \nabla u\cdot\nabla[(x\cdot\nabla u)\psi_n] - \int_{\partial\Omega\setminus\{0\}} \frac{\partial u}{\partial\nu}(x\cdot\nabla u)\psi_n=\int_\Omega\vert u\vert^{p-1}u(x\cdot\nabla u)\psi_n = \int_\Omega x\cdot\nabla \frac{\vert u\vert^{p+1}}{p+1} \psi_n.
$$
So,
\begin{multline*}
\int_\Omega \nabla u\cdot\nabla(x\cdot\nabla u)\psi_n + \int_\Omega \nabla u\cdot\nabla \psi_n (x\cdot\nabla u)- \int_{\partial\Omega\setminus\{0\}} \frac{\partial u}{\partial\nu}(x\cdot\nabla u)\psi_n=\\
-\int_\Omega \frac{\vert u\vert^{p+1}}{p+1}(d\psi_n+x\cdot\nabla \psi_n) + \int_{\partial\Omega\setminus\{0\}} \frac{\vert u\vert^{p+1}}{p+1}(x\cdot\nu)\psi_n
=-\frac d{p+1}\int_\Omega \vert u\vert^{p+1}\psi_n -\int_\Omega \frac{\vert u\vert^{p+1}}{p+1}(x\cdot\nabla\psi_n).
\end{multline*}
It follows that
$$
\int_\Omega \vert\nabla u\vert^2\psi_n + \int_\Omega \left(x\cdot\nabla\frac{\vert\nabla u\vert^2}{2}\right)\psi_n +\int_\Omega \nabla u\cdot\nabla \psi_n (x\cdot\nabla u)- \int_{\partial\Omega\setminus\{0\}} \frac{\partial u}{\partial\nu}(x\cdot\nabla u)\psi_n = -\frac d{p+1}\int_\Omega \vert u\vert^{p+1}\psi_n -\int_\Omega \frac{\vert u\vert^{p+1}}{p+1}(x\cdot\nabla\psi_n),
$$
i.e.
\begin{multline*}
    \int_\Omega \vert\nabla u\vert^2\psi_n - \int_\Omega \frac{\vert\nabla u\vert^2}{2}(d\psi_n+x\cdot\nabla\psi_n)+\int_{\partial\Omega\setminus\{0\}} \frac{\vert\nabla u\vert^2}{2}(x\cdot\nu)\psi_n+\\
    \int_\Omega \nabla u\cdot\nabla \psi_n (x\cdot\nabla u)- \int_{\partial\Omega\setminus\{0\}} \frac{\partial u}{\partial\nu}(x\cdot\nabla u)\psi_n = -\frac d{p+1}\int_\Omega \vert u\vert^{p+1}\psi_n -\int_\Omega \frac{\vert u\vert^{p+1}}{p+1}(x\cdot\nabla\psi_n).
\end{multline*}
On $\partial\Omega\setminus\{0\}$, $\nabla u = u_\nu\nu$ and $0=r\partial_r u =x\cdot\nabla u$, since $u=0$ on $\partial\Omega$ and $\Omega$ is a cone. Thus, $x\cdot\nabla u=u_\nu(x\cdot\nu)$ and $\frac{\vert\nabla u\vert^2}2(x\cdot\nu)=\frac{u_\nu^2}2(x\cdot\nu)=0$ on $\partial\Omega\setminus\{0\}$ so both boundary integrals are equal to zero! 
It follows that
$$
(1-d/2)\int_\Omega \vert\nabla u \vert^2\psi_n + \frac d{p+1}\int_\Omega \vert u\vert^{p+1}\psi_n =
\int_\Omega \frac{\vert\nabla u\vert^2}{2}(x\cdot\nabla\psi_n) - \int_\Omega \nabla u\cdot\nabla\psi_n(x\cdot\nabla u) -\int_\Omega \frac{\vert u\vert^{p+1}}{p+1}(x\cdot\nabla\psi_n)=:I(n).
$$
For $n\ge1$, the integrands of the right-hand side can be bounded as follows
$$
\left\vert
\frac{\vert\nabla u\vert^2}{2}(x\cdot\nabla\psi_n)-\nabla u\cdot\nabla\psi_n(x\cdot\nabla u) -\frac{\vert u\vert^{p+1}}{p+1}(x\cdot\nabla\psi_n)
\right\vert\le
\left(\frac32\vert\nabla u\vert^2+\frac{\vert u\vert^{p+1}}{p+1}
\right)\vert x\vert\cdot\vert\nabla\psi_n\vert=:g\vert x\vert\cdot\vert\nabla\psi_n\vert\qquad\text{in $\Omega,$} 
$$
where $g\in L^1(\Omega)$ and $g\ge0$. Taking the truncation functions of the form $\theta_n(x)=\theta(x/n)$ and $\zeta_n(x)=\zeta(nx)$ if $d\ge 3$ (and using log-type capacitary truncation functions if $d=2$), we infer that
$$
\int_\Omega g \vert x\vert\cdot\vert\zeta_n\nabla\theta_n\vert \le
\Vert\nabla \theta\Vert_\infty\int_{\Omega\cap[n<\vert x\vert<2n]}g\frac{\vert x\vert}n\le 2\Vert\nabla\theta\Vert_\infty\int_{\Omega\cap[n<\vert x\vert<2n]}g\to0
$$
as $n\to+\infty$ and similarly
$$
\int_\Omega g \vert x\vert\cdot\vert\theta_n\nabla\zeta_n\vert \le 2\Vert\nabla\zeta\Vert_\infty\int_{\Omega\cap[1/n<\vert x\vert<2/n]}g\to0.
$$
Since $\nabla u\in L^2(\Omega)$ and $u\in L^{p+1}(\Omega)$, we may apply the dominated convergence theorem to pass to the limit in $(1-d/2)\int_\Omega\vert\nabla u\vert^2\psi_n-\frac d{p+1}\int_\Omega \vert u\vert^{p+1}\psi_n$ and the proposition follows.
\end{proof}
Next, in order to exploit Poho\v{z}aev's identity, we establish the following elementary identity.
\begin{prop}\label{prop2}
Under the assumptions of Lemma \ref{l2bis}, there holds 
$$
\int_\Omega \vert\nabla u\vert^2 = \int_\Omega \vert u\vert^{p+1}.
$$
\end{prop}
\begin{proof}
Let $(\theta_n)$ be a standard truncation sequence near infinity and apply \eqref{eqfaibleplus} with $\psi=\theta_n$ to get
$$
\int_\Omega\nabla u\cdot\nabla (u\theta_n)= \int_\Omega \vert u\vert^{p+1}\theta_n.
$$
The left-hand side equals
$$
\int_\Omega \vert\nabla u\vert^2\theta_n + \int_\Omega \nabla\frac{u^2}2\cdot\nabla\theta_n = \int_\Omega \vert\nabla u\vert^2\theta_n - \int_\Omega \frac{u^2}2\cdot\Delta\theta_n , 
$$
where we used that $u=0$ on $\partial\Omega$ and $\nabla\theta_n=0$ in $B_1 \cup \R^d \setminus B_{2n}$. 
In addition,
$$
\left\vert\int_\Omega {u^2}\Delta\theta_n\right\vert \le
\left(\int_{\Omega\cap[n<\vert x\vert<2n]}\vert u\vert^{p+1}\right)^{\frac2{p+1}}\left(\int_{\Omega\cap[n<\vert x\vert<2n]}\vert\Delta\theta_n\vert^{\frac{p+1}{p-1}}\right)^{\frac{p-1}{p+1}},
$$
and since $p\le p_S(d)$,
$$
\int_{\Omega\cap[n<\vert x\vert<2n]}\vert\Delta\theta_n\vert^{\frac{p+1}{p-1}} \le C n^{d-2\frac{p+1}{p-1}}\to0,
$$
as $n\to\infty$.
Since $u\in L^{p+1}(\Omega)$ and $\nabla u \in L^2(\Omega)$, we may pass to the limit and obtain the desired conclusion.
\end{proof}
\begin{remark}\label{rem1}
Combining Propositions \ref{prop1} and \ref{prop2}, we immediately conclude that $u=0$ if $p<p_S(d)$.
\end{remark}
To deal with the critical case $p=p_S(d)$, we prove the following
\begin{prop}\label{prop3} Let $p=p_S(d)$ and $u\in C^2(\Omega)\cap C(\overline\Omega)$ be a classical solution of \eqref{main} in a cone $\Omega$ such that $\partial\Omega$ is a graph in the $x_d$ direction. Then, under the assumptions of Lemma \ref{l2bis}, there holds
$$
\int_{\partial\Omega\setminus\{0\}}\nu_d\vert\nabla u \vert^2 =0,
$$
where $\nu_d$ is the $d$-th coordinate of the outward unit normal vector $\nu$.
\end{prop}
\begin{proof}
We multiply \eqref{main} with the test function $\frac{\partial u}{\partial x_d}\psi_n$ (where $\psi_n$ has been defined at the 
beginning of the proof of Proposition \ref{prop1}) and integrate by parts. The left-hand side equals
$$
-\int_\Omega \Delta u\left(\frac{\partial u}{\partial x_d}\psi_n\right) = \int_\Omega \nabla u\cdot\nabla\left(\frac{\partial u}{\partial x_d}\psi_n\right) -\int_{\partial\Omega\setminus\{0\}}\frac{\partial u}{\partial \nu}\frac{\partial u}{\partial x_d}\psi_n
$$
while the right-hand side equals
$$
\int_\Omega \vert u\vert^{p-1}u\frac{\partial u}{\partial x_d}\psi_n = 
\int_\Omega \frac{\partial}{\partial x_d}\left(\frac{\vert u\vert^{p+1}}{p+1}\right)\psi_n=
-\int_\Omega \frac{\vert u\vert^{p+1}}{p+1}\frac{\partial\psi_n}{\partial x_d},
$$
where we used the fact that $u=0$ on $\partial\Omega\setminus\{0\}$. Hence,
$$
\int_\Omega \nabla u\cdot\nabla\left(\frac{\partial u}{\partial x_d}\right)\psi_n + \int_\Omega \nabla u\cdot\nabla\psi_n \frac{\partial u}{\partial x_d} - \int_{\partial\Omega\setminus\{0\}}\frac{\partial u}{\partial\nu}\frac{\partial u}{\partial x_d}\psi_n=-\int_\Omega \frac{\vert u\vert^{p+1}}{p+1}\frac{\partial\psi_n}{\partial x_d}.
$$
Since $\nabla u\cdot\nabla\left(\frac{\partial u}{\partial x_d}\right)=\frac{\partial}{\partial x_d}\frac{\vert\nabla u\vert^2}2$, we integrate by parts the first term and find
$$
-\int_\Omega  \frac{\vert \nabla u\vert^2}2\frac{\partial \psi_n}{\partial x_d} +\int_{\partial\Omega\setminus\{0\}} \frac{\vert\nabla u\vert^2}{2}\psi_n\nu_d+ \int_\Omega \nabla u\cdot\nabla\psi_n \frac{\partial u}{\partial x_d} - \int_{\partial\Omega\setminus\{0\}}\frac{\partial u}{\partial\nu}\frac{\partial u}{\partial x_d}\psi_n=-\int_\Omega \frac{\vert u\vert^{p+1}}{p+1}\frac{\partial\psi_n}{\partial x_d}.
$$
Since $\nabla u =u_\nu\nu$ on $\partial\Omega\setminus\{0\}$, $\frac{\partial u}{\partial x_d} = u_\nu\nu_d$ on $\partial\Omega\setminus\{0\}$ and so 
$\frac{\partial u}{\partial \nu}\frac{\partial u}{\partial x_d}= u_\nu^2\nu_d=\vert\nabla u\vert^2\nu_d$ on $\partial\Omega\setminus\{0\}$. We deduce that
$$
I(n):=-\int_\Omega \frac{\vert\nabla u\vert^2}2\frac{\partial\psi_n}{\partial x_d}+\int_\Omega \nabla u\cdot\nabla\psi_n\frac{\partial u}{\partial x_d}+
\int_\Omega \frac{\vert u\vert^{p+1}}{p+1}\frac{\partial\psi_n}{\partial x_d} = \frac12\int_{\partial\Omega\setminus\{0\}}\vert\nabla u\vert^2\psi_n\nu_d.
$$
By the graph property of $\partial\Omega$, $-\nu_d\ge0$. So, in order to conclude, it suffices to prove that the left-hand side in the above identity converges to $0$ as $n\to+\infty$ (using Beppo Levi's theorem on $\partial\Omega\setminus\{0\}$). But
$$
\vert I(n)\vert\le \int_\Omega \left(\frac32\vert\nabla u\vert^2+\frac{\vert u\vert^{p+1}}{p+1}\right)\vert\nabla\psi_n\vert\to 0,
$$
as $n\to+\infty$, as in the proof of Proposition \ref{prop1}.
\end{proof}

\begin{remark}\label{rem2}
 If $\partial\Omega$ is a graph, then there exists $p\in\partial\Omega\setminus\{0\}$ such that $\nu_d(p)\neq0$ and so there exists $\rho\in(0,1)$ such that $\vert\nabla u\vert^2=0$ in $(\partial\Omega\setminus\{0\})\cap B(p,\rho)$. By the unique continuation principle, $u=0$ in $\Omega$.
\end{remark}
Going back to the statement of Theorem \ref{To prove:SOAC2}, our cone $\Omega$ is star-shaped with respect to the north pole, so thanks to Lemma \ref{geometric1} below, the boundary of the cone $\Omega$ is a graph in the $x_d$ direction. Then we may apply Proposition \ref{prop3} and by Remarks \ref{rem1} and \ref{rem2}, we have just proved Theorems \ref{To prove:SOAC}, \ref{To prove:SOAC2} and Corollary \ref{lastcor} in the case $1<p\le p_S(d)$. Corollary \ref{cor:SOAC} also follows for $1<p\le p_S(d)$ since in that case and for $\Omega=\R^d_+$, any weak solution is in fact a classical solution $u\in C^2(\super{\R^d_+})$.
\begin{remark}Regarding weak solutions, we observe that Lemma \ref{l2bis} remains valid for weak solutions stable outside a compact set.
Indeed, in these cases, \eqref{eqfaibleplus} holds and so also does Lemma \ref{l2bis}.

\end{remark}
We turn next to 
\subsection{The supercritical case $p>p_S(d)$}

\

Observe that if $u$ solves \eqref{main} and $\lambda>0$, then the function $u_\lambda$ defined by
\begin{equation}\label{rescaled solutions}
    u_\lambda(x):=\lambda^{\frac{2}{p-1}}u(\lambda\,x),  \qquad x\in \Omega
\end{equation}
is also a solution to \eqref{main}. In order to understand the asymptotic profile of $u$ at infinity, we consider the blow-down family $(u_\lambda)$ as $\lambda\to+\infty$.

\subsubsection{\textit{\textbf{A priori}} estimates and convergence of the blow-down family}

\begin{lemma}\label{energy} Let $p\ge p_S(d)$. Assume that $u$ is a weak solution of \eqref{main} stable outside a compact set. Then, there exist constants $C=C(u,d,p)> 0 $ 
such that for all 
$R>1$ and $\lambda\ge 1$,
\begin{equation}\label{bound2}
      \int_{\Omega_R} \vert\nabla u_\lambda\vert^2+\vert u_\lambda\vert^{p+1}\leq CR^{d-2\frac{p+1}{p-1}}.
\end{equation}
\end{lemma}

This follows essentially from \cite{Farina}. For completeness, we include a proof below.

\begin{proof}
Fix $R_0 = R_0(u) >1$ such that $u$ is stable outside of $\Omega_{R_0}$.
Then, given $\lambda\ge 1$, $u_\lambda$ is stable outside the smaller domain $\Omega_{R_0/\lambda}$. Therefore, by proceeding as in the proof of Lemma \ref{l2bis}, we obtain inequalities \eqref{bound0} and \eqref{bound1} for $u_\lambda$. Hence, for $R >2R_0/\lambda$,  
$$
\int_{\Omega_R \setminus \Omega_{2R_0/\lambda}} \,\vert\nabla u_\lambda\vert^2+\vert u_\lambda\vert^{p+1} \le 
C_{p,d} \big ((R_0/\lambda)^{d-2\frac{p+1}{p-1}}+R^{d-2\frac{p+1}{p-1}} \big) \le 2 C_{p,d} R^{d-2\frac{p+1}{p-1}} 
$$
where we have used that $p \ge p_S(d)$ implies $d-2\frac{p+1}{p-1}\ge 0$.

On the other hand we also have  
\begin{equation}\label{bound4}
\int_{\Omega_{2R_0/\lambda}}\,\vert\nabla u_\lambda\vert^{2}+u_\lambda^{p+1} = \lambda^{2\frac{p+1}{p-1}-d}\int_{\Omega_{2R_0}}\,\vert\nabla u\vert^{2}+\vert u\vert^{p+1}\le \int_{\Omega_{2R_0}}\,\vert\nabla u\vert^{2}+\vert u\vert^{p+1} = C(u,p,R_0(u)) < +\infty
\end{equation}
since $\lambda\ge 1$. The lemma follows.\end{proof}

\begin{remark}\label{optimal energy}
    Let $p>p_S(d)$. Assume that $u$ is a weak solution of \eqref{main} stable outside a compact set. Then, the following improved version of \eqref{bound2} holds for some $\gamma>1$ and all $ R>1$
    \begin{equation}\label{bound3}
         \int_{\Omega_R} \vert \nabla (\vert u_\lambda \vert^{\frac{\gamma-1}{2}} u_\lambda) \vert^2+\vert u_\lambda\vert^{p+\gamma}\leq C(1+R^{d-2\frac{p+\gamma}{p-1}}),
    \end{equation}
    where $C$ is constant depending only on $ p,d, \gamma$ and $u$.
    
    Precisely, as in Proposition 4 of \cite{Farina}, testing equation \eqref{main} with\footnote{For any $\beta>0$, the function $\vert T_k(u)\vert^{\beta} T_k(u) \varphi^2$ belongs to $H^1_0(\Omega) \cap L^{p+1}(\Omega)$ and has bounded support. Indeed, pick $\xi \in C^1_c(\R^d)$ such that $\xi \equiv 1 $ on a neighbourhood of the support of $\varphi$, then $u \xi \in H^1_0(\Omega)$. The latter implies $\vert T_k(u \xi)\vert^{\beta} T_k(u \xi) \in H^1_0(\Omega)$ and so also $\vert T_k(u \xi)\vert^{\beta} T_k(u \xi) \varphi^2$ belongs to $H^1_0(\Omega)$. Then observe that $\vert T_k(u \xi)\vert^{\beta} T_k(u \xi) \varphi^2 = \vert T_k(u)\vert^{\beta} T_k(u) \varphi^2$ on $\Omega$. Since it is clear that $\vert T_k(u)\vert^{\beta} T_k(u) \varphi^2$ belongs to $ L^{p+1}(\Omega)$ and has bounded support, the desired claim follows.

    } $\vert T_k(u)\vert^{\gamma-1} T_k(u) \varphi^2$, where $T_k(u)$ is the truncation of $u$ at level $k>0$, $\gamma\in(1,2p+2\sqrt{p(p-1)}-1)$ and $\varphi \in C^1_c(\R^d)$, and then inserting $\vert T_k(u)\vert^{\frac{\gamma-1}{2}} T_k(u) \varphi^2$ in the stability condition, we are lead to bounds \eqref{bound0} and \eqref{bound1}, where the exponent $d-2\frac{p+1}{p-1}$ is replaced by $d-2\frac{p+\gamma}{p-1}$. Finally, choosing $\gamma$ such that $d-2\frac{p+\gamma}{p-1} \ge 0$ (and filling the hole) leads to the desired estimate \eqref{bound3}.
\end{remark}

We use the preceding remark to prove convergence of $(u_\lambda)$ to a limit $u_\infty$ in suitable spaces.

\begin{lemma}\label{convergence}
    Let $p>p_S(d)$. Assume that $u$ is a weak solution of \eqref{main} stable outside a compact set. The rescaled solutions converge along a sequence $(u_{\lambda_n})$ to a limit $u_\infty$ strongly 
    in $H^1\cap\,L^{p+1}(\Omega_R)$ for every $R>0$, as $\lambda_n\to\infty$. Furthermore, $u_\infty$ is a weak, stable solution of \eqref{main}. 
\end{lemma}

\begin{proof} 
By estimate \eqref{bound2} and \eqref{bound3} there is $\gamma>1$ such that $(u_\lambda)$ is bounded in $H^1 \cap\,L^{p+\gamma}(\Omega_R)$ for every $R>0$. 
Thus, by a standard diagonal argument, there exists a measurable function $u_\infty$ defined on $\Omega$ and a subsequence $(u_{\lambda_n})$ weakly converging to $u_\infty$ in $H^1\cap\,L^{p+\gamma}(\Omega_R)$ for every $R>0$. In particular we get 
\begin{equation}\label{cvfaible1}
       \int_{\Omega}\nabla\,u_{\lambda_n}\nabla \eta \to
       \int_{\Omega}\nabla\,u_\infty\nabla\eta,
\end{equation}    
for all $\eta \in \, H^1_0(\Omega)$ with bounded support.

We also observe that, by Rellich-Kondrachov's compactness theorem, we may and do suppose that the convergence is strong in $L^2(\Omega_R)$, for every $R>0$.\footnote{Indeed, let $(\phi_k)$ be a sequence of standard cut-off functions of class $C^1_c(\R^d)$. Then, for every fixed $k$, the sequence $(u_{\lambda_n} \phi_k)$ is bounded in $H^1_0 (\Omega_{3k})$ and by Rellich-Kondrachov's compactness theorem, it is compact in $L^2(\Omega_{3k})$ which, in turn, implies that the sequence $(u_{\lambda_n})$ is compact in $L^2(\Omega_{k})$. The desired conclusion then follows by a standard diagonal argument.}
To obtain strong convergence in $L^{p+1}(\Omega_R)$, we apply the following interpolation inequality 
    \begin{equation*}
       \norm{u_\lambda-u_\infty}_{L^{p+1}(\Omega_R)}\leq\norm{u_\lambda-u_\infty}^\theta_{L^{2}(\Omega_R)}\norm{u_\lambda-u_\infty}^{1-\theta}_{L^{p+\gamma}(\Omega_R)},
    \end{equation*} 
where $\theta=\frac{2\gamma-2}{(p+\gamma-2)(p+1)}$, as follows from H\"{o}lder's inequality. Since $(u_{\lambda_n})$ converges to $u_\infty$ in $L^2(\Omega_R)$ and $(u_{\lambda_n})$ is bounded in $L^{p+\gamma}(\Omega_R)$, it follows that $(u_{\lambda_n})$ converges strongly to $u_\infty$ in $L^{p+1}(\Omega_R)$ for every $R>0$. Therefore, we can find a subsequence (still denoted by $(u_{\lambda_n}$)) and $g \in L^{p+1}(\Omega_R)$ for every $R>0$ such that 
\begin{equation}\label{pointwisecvplus}
\begin{split}
       &\vert u_{\lambda_n} \vert \leq g \qquad a.e. \,\, in \,\, \Omega,\\
       & u_{\lambda_n}\to u_\infty \qquad a.e. \,\, in \,\, \Omega.
    \end{split}
\end{equation}    
The latter and Lebesgue's dominated convergence theorem imply that 
\begin{equation}\label{cvfaible2}
        \int_{\Omega}\,|u_{\lambda_n}|^{p-1}u_{\lambda_n}\eta \to \int_{\Omega}|u_\infty|^{p-1}u_\infty\eta,
    \end{equation}
for all $\eta \in \, L^{p+1}(\Omega)$ with bounded support. 

Recalling that $u_{\lambda_n}$ is a weak solution to \eqref{main}, we see that the combination of \eqref{cvfaible2} and \eqref{cvfaible1} gives

\begin{equation}\label{weak-eq-infty}
        \int_{\Omega}\nabla\,u_\infty\nabla\eta =\int_{\Omega}|u_\infty|^{p-1}u_\infty\eta,
\end{equation}
for all $\eta \in H^1_0(\Omega) \cap L^{p+1}(\Omega)$ with bounded support.


Next we observe that $u_{\infty} \zeta\in H^1(\Omega)$ for every $\zeta\in C^1_c(\R^d)$ and that $(u_{\lambda_n} \zeta)$ converges weakly to $u_\infty \zeta$ in $H^1(\Omega)$, thanks to the weak convergence of $(u_{\lambda_n})$ in $H^1(\Omega_R)$ for every $R>0$. Hence, $u_{\infty} \zeta\in H^1_0(\Omega)$ for every $\zeta\in C^1_c(\R^d)$, since $H^1_0(\Omega)$ is weakly closed in $H^1(\Omega)$.  
The latter and \eqref{weak-eq-infty} prove that $u_{\infty}$ is a weak solution to \eqref{main} and 
\begin{equation}\label{eqfaibleplus-infty}
      \int_{\Omega}\nabla u_{\infty} \nabla (u_{\infty}\varphi^2) =  \int_{\Omega} \vert u_{\infty} \vert^{p+1} \varphi^2 \qquad \forall \, \varphi \in C^1_c(\R^d)
\end{equation}
since $\,u_{\infty} \varphi^2$ belongs to $H^1_0(\Omega) \cap\,L^{p+1}(\Omega)$ and has bounded support. 

 Next, we claim that $u_{\lambda_n}\varphi \to u_\infty \varphi$ in $H^1_0(\Omega)$. In view of \eqref{eqfaibleplus-infty} and 
 \eqref{weak sol} we can compute as in the first line of \eqref{weak sol2} and find that, for every $\varphi \in C^1_c(\R^d)$ 
    \begin{align*}
        \int_{\Omega}
        |\nabla(u_{\lambda_n}\varphi)|^2&=\int_{\Omega} 
        |u_{\lambda_n}|^2|\nabla\varphi|^2+\int_{\Omega}\,
        |u_{\lambda_n}|^{p+1}\varphi^2
        \to\int_{\Omega} 
        |u_\infty|^2|\nabla\varphi|^2+\int_{\Omega}\,
        |u_\infty|^{p+1}\varphi^2
        =\int_{\Omega}
        |\nabla(u_\infty\varphi)|^2,
    \end{align*}
where in the latter we have used \eqref{pointwisecvplus} and 
Lebesgue's dominated convergence theorem. It follows that 
$(u_{\lambda_n}\varphi)$ converges to $u_\infty \varphi$ in $H^1_0(\Omega)$, for every $\varphi \in C^1_c(\R^d)$ and, in particular $u_{\lambda_n} \to u_\infty$ in $H^1(\Omega_R)$ for every $R>0$.
Stability of the weak solution $u_\infty$ follows from the strong convergence of $(u_{\lambda_n})$ in $L^{p+1}(\Omega_R)$ 
and the fact that for fixed $n$, $u_{\lambda_n}$ is stable outside $K/\lambda_n$ (here $K \subset \Omega$ is any compact set such that $u$ is stable outside $K$).
\end{proof}

\begin{remark}
    One may wonder if Lemma \ref{convergence} still holds if $p=p_S(d)$. If one works in all of $\R^d$, the answer is no. Indeed, a constant multiple of the function $u(x)=(1+\vert x\vert^2)^{-\frac{d-2}2}$ solves the equation and is stable outside a compact set, thanks to Hardy's inequality. Yet, the family $(u_\lambda)$ is clearly not compact in $L^{p+1}(B_1)$.
\end{remark}

\subsubsection{Monotonicity formula}
Define the functional 
\begin{equation}\label{E}
E(u;\lambda) = 
\lambda^{2\frac{p+1}{p-1}-d}\int_{\Omega_\lambda} \left(\frac12\vert\nabla u\vert^2 - \frac{1}{p+1}\vert  u\vert^{p+1}\right)\;dx
+ \lambda^{\frac{4}{p-1}+1-d}\frac{1}{p-1}\int_{\Omega\cap\partial B_{\lambda}} u^2\;d\sigma,
\end{equation}
where $\lambda>0$, $\Omega_{\lambda}=\Omega\cap B_{\lambda}$ and $u\in C^2(\Omega)\cap C(\overline\Omega)$ is a classical solution to \eqref{main}. 
\begin{remark}
Note that $E$ satisfies the following simple scaling relation: given $\lambda,R>0$, 
$$E(u;\lambda R)=E(u_{\lambda};R)$$
\end{remark}
The main result of this section is the following monotonicity formula, which extends a result due to Pacard \cite{p}. 
\begin{lemma}\label{monotonicity formula}
   Let $u\in C^2(\Omega)\cap C(\super\Omega)$ be a solution of \eqref{main} 
   and $\lambda>0$. Let $E$ be as in \eqref{E}. Then $\lambda\mapsto E(u;\lambda)$ is a non-decreasing function. Furthermore, $E(u;\cdot)\in C^1(\R_+^*)$ and for all $\lambda>0$,
\begin{equation}\label{derivative of E}
\frac{d E}{d\lambda}(u;\lambda) = \lambda^{\frac{4}{p-1}+2-d}\int_{\Omega\cap\partial B_\lambda}\left({\partial_r  u}+\frac{2}{p-1}\frac { u}r\right)^2\;d\sigma.
\end{equation}
\end{lemma}
\proof Fix $\eps \in (0,1)$. Let
\begin{equation} \label{E1}
E_{1}^\eps(u;\lambda) = 
\lambda^{2\frac{p+1}{p-1}-d}\int_{\Omega_\lambda\setminus \Omega_{\lambda\eps}} \left(\frac12\vert\nabla u\vert^2 - \frac{1}{p+1}\vert  u\vert^{p+1}\right)\;dx.
\end{equation} 
For $\lambda>0$, let also $u_\lambda$ be defined by \eqref{rescaled solutions}. Then, $u_\lambda$ satisfies the three following properties: given $\lambda>0$, $u_\lambda$ solves \eqref{main},
\begin{equation}\label{t3}  
E_{1}^\eps(u;\lambda) = E_{1}^\eps(u_\lambda; 1),
\end{equation} 
and
\begin{equation} \label{t4} 
\lambda {\partial_\lambda u_{\lambda}} = \frac{2}{p-1} u_\lambda + r{\partial_r u_{\lambda}}\quad\text{ for every $(x,\lambda)\in\Omega\times\R_+^*$}.
\end{equation} 
By standard elliptic regularity, $u_\lambda\in C^2(\super{\Omega_{1}\setminus\Omega_{\eps}})$ . So $\lambda\mapsto E^\eps_{1}(u;\lambda)$ is $C^1$ 
and differentiating the right-hand side of \eqref{t3}, we find
$$
\frac{dE_{1}^\eps}{d\lambda}(u;\lambda) = \int_{\Omega_1\setminus\Omega_{\eps}}\nabla u_\lambda\cdot\nabla \partial_\lambda u_{\lambda}\;dx -\int_{\Omega_1\setminus\Omega_{\eps}}\vert u_\lambda\vert^{p-1}u_{\lambda}\partial_\lambda u_{\lambda}\;dx.
$$ 
Integrating by parts, using the equation, the boundary condition and the fact that $\Omega$ is a cone, we find 
\begin{align*}
\frac{dE_{1}^\eps}{d\lambda}(u;\lambda) &= \int_{\Omega\cap\partial B_1}\partial_r u_{\lambda}\partial_\lambda u_{\lambda}\;d\sigma-\int_{ \Omega\cap\partial B_{\eps}}\partial_r u_{\lambda}\partial_\lambda u_{\lambda}\;d\sigma.
\end{align*}
Using \eqref{t4},
\begin{align*}
\left\vert\int_{ \Omega\cap\partial B_{\eps}}\partial_r u_{\lambda}\partial_\lambda u_{\lambda}\;d\sigma\right\vert&=
\frac1\lambda\left\vert\int_{ \Omega\cap\partial B_{\eps}}\frac2{p-1}u_\lambda\partial_r u_{\lambda}+\eps (\partial_ru_{\lambda})^2\;d\sigma\right\vert\\
&\le C \|u_\lambda \|_{L^2( \Omega\cap\partial B_{\eps})}  \|\nabla u_\lambda \|_{L^2( \Omega\cap\partial B_{\eps})} + \eps \|\nabla u_\lambda \|_{L^2( \Omega\cap\partial B_{\eps})}^2\\
&\le C\eps \|\nabla u_\lambda \|_{L^2( \Omega\cap\partial B_{\eps})}^2,
\end{align*}
where we used (the sharp) Poincaré inequality in $H^1_0(\partial B_\eps\cap\Omega)$ in the last inequality.
Since $u_\lambda\in H^1(\Omega_{1})$, $\liminf_{\eps\to0}{(\eps\vert\ln\eps\vert)}\|\nabla u_\lambda \|^2_{L^2( \Omega\cap\partial B_{\eps})}=0$ and so there exists a sequence $\eps_{n}\to 0$ such that 
$$\|\nabla u_\lambda \|^2_{L^2( \Omega\cap\partial B_{\eps_{n}})}\le\frac 1{\eps_n\vert\ln\eps_n\vert}.$$ 
And so
\begin{align*}
C|\ln\eps_n|^{-1} &\ge \left\vert \frac{dE_{1}^{\eps_{n}}}{d\lambda}(u;\lambda)- \int_{\Omega\cap\partial B_1}\partial_r u_{\lambda}\partial_\lambda u_{\lambda}\;d\sigma\right\vert\\
&\ge \left\vert \frac{dE_{1}^{\eps_{n}}}{d\lambda}(u;\lambda)-\lambda \int_{\Omega\cap\partial B_1}(\partial_\lambda u_{\lambda})^2\;d\sigma + \frac{2}{p-1}\int_{\Omega\cap\partial B_1}u_\lambda(\partial_\lambda u_{\lambda})\;d\sigma\right\vert \\
&\ge \left\vert \frac{d}{d\lambda}\left(E_{1}^{\eps_{n}}(u;\lambda)+\frac{1}{p-1}\int_{\Omega\cap\partial B_1} u_\lambda^2\;d\sigma\right)-\lambda \int_{\Omega\cap\partial B_1} (\partial_\lambda u_{\lambda})^2\;d\sigma \right\vert.\\
\end{align*}
For any fixed $\lambda>0$, take $0<\lambda_1<\lambda_2\le\lambda$ and integrate the above inequality between $\lambda_1$ and $\lambda_2$. Then, letting $\epsilon_n\to0$, we find that
\begin{equation}\label{mff}
E(u;\lambda_2)-E(u;\lambda_1) = \int_{\lambda_1}^{\lambda_2}\lambda \int_{\Omega\cap\partial B_1} (\partial_\lambda u_{\lambda})^2\;d\sigma\, d\lambda.
\end{equation}

Hence, $\lambda\mapsto E(u;\lambda)\in C^1(\R_+^*)$ and scaling back the above quantity, the lemma follows.\hfill\qed

\begin{remark}
Let $u\in C^2(\Omega)\cap C(\super\Omega)$ be a solution of \eqref{main}. If the function $\lambda \mapsto E(u;\lambda)$ is constant on $(0, +\infty)$, then by \eqref{derivative of E} 
it follows that $u$ is homogeneous of degree $-\frac{2}{p-1}$.
\end{remark}

\begin{remark}\label{form-monotonia-coni-troncati-0}
The proof of Lemma \ref{monotonicity formula} immediately extends to solutions defined on $\Omega_R$, $R>0$ i.e. let $R>0$ and $u\in C^2(\Omega_R)\cap C(\super\Omega_R)$ be a solution of 
\begin{equation} \label{main-troncato}
\left\{
\begin{aligned}
-\Delta u &= \vert u\vert^{p-1}u&\quad\text{in $\Omega_R$,}\\
u&=0&\quad\text{on $\partial\Omega \cap B_R$.}
\end{aligned}
\right. 
\end{equation}
For $\lambda \in (0,R)$ let $E$ be as in \eqref{E}. Then $\lambda\mapsto E(u;\lambda)$ is a non-decreasing function. Moreover, $E(u;\cdot)\in C^1(0,R)$ and, for all $\lambda \in (0,R)$,
\begin{equation}\label{derivative of E-troncata}
\frac{d E}{d\lambda}(u;\lambda) = \lambda^{\frac{4}{p-1}+2-d}\int_{\Omega\cap\partial B_\lambda}\left({\partial_r  u}+\frac{2}{p-1}\frac { u}r\right)^2\;d\sigma.
\end{equation}
and for any $0<\lambda_1<\lambda_2 < R $ we find that
\begin{equation}\label{mff-coni-troncati-0}
E(u;\lambda_2)-E(u;\lambda_1) = \int_{\lambda_1}^{\lambda_2}\lambda \int_{\Omega\cap\partial B_1} (\partial_\lambda u_{\lambda})^2\;d\sigma\, d\lambda.
\end{equation}
\end{remark}

\begin{remark}\label{form-monotonia-coni-troncati} The monotonicity formula remains valid for weak stable solutions in the sense that the function defined on $\R_+^*$ by $\lambda \mapsto E(u;\lambda)$ is monotone, absolutely continuous and its derivative is given by \eqref{derivative of E} for a.e. $\lambda\in(0,+\infty)$. 

Indeed, if $u$ is weak stable solution of \eqref{main}, then, by Proposition 13 in \cite{df}, for every $R>0$, there exists a sequence of non-negative solutions $u_n^R\in C^2(\super{B_R^+})$ such that $u_n^R=0$ on $\partial\R^d_+\cap B_R$, $u_n^R\longrightarrow u$ a.e. in $\Omega_R = B_R^+$, and $u_n^R\to u$ in $H^1(\Omega_R)$, as $n\to+\infty$. By Remark \ref{form-monotonia-coni-troncati-0},
$\lambda\mapsto E(u_n^R;\lambda)$ is a non-decreasing function of $\lambda \in (0,R)$, $E(u_n^R;\cdot)\in C^1(0,R)$ and for any $0<\lambda_1<\lambda_2 < R $ we find that
\begin{equation}\label{mff-coni-troncati}
E(u_n^R;\lambda_2)-E(u_n^R;\lambda_1) = \int_{\lambda_1}^{\lambda_2}\lambda \int_{\Omega\cap\partial B_1} (\partial_\lambda (u_n^R)_{\lambda})^2\;d\sigma\, d\lambda.
\end{equation}
Passing to the limit in \eqref{mff-coni-troncati} as $n\to+\infty$, our claim follows.


\end{remark}

\subsubsection{Analysis of the limiting profile $u_\infty$}\label{consequences}

In this section, we will show that $u_\infty$ (given in Lemma \ref{convergence}) is a homogeneous function. 
To do so, first note that given a weak solution $u$ of \eqref{main}, if $u_\lambda=u$ for all $\lambda>0$, then $u$ is homogeneous of degree $-\frac{2}{p-1}$, and so it must have the form
\begin{equation}\label{sing solution}
    u(r,\theta)=r^{-\frac{2}{p-1}}v(\theta),
\end{equation}
for some measurable function $v\colon A\subset\mathbb{S}^{d-1}\to\R$. This is a simple consequence of taking $\lambda=|x|^{-1}$ in \eqref{rescaled solutions}. Note that by \eqref{sing solution}, $u_\lambda=u$ automatically excludes the possibility that $u$ is a classical solution of \eqref{main} because of the singularity near $x=0$ (unless $v=0$).

Now we prove an important property of the limit function $u_\infty$ given in Lemma \ref{convergence}.
\begin{lemma}\label{properties of uinf}
    Let $p>p_S(d)$ and assume that $u\in C^2(\Omega)\cap C(\super\Omega)$ is a solution of \eqref{main} stable outside a compact set. 
    The blow-down limit $u_\infty$ is homogeneous of degree $-\frac{2}{p-1}$ and thus has the form \eqref{sing solution} for a suitable $v\in\,H^1_0\cap\,L^{p+1}(A)$.
\end{lemma}
\begin{proof}
After fixing two radii $R_2>R_1>0$ we set 
\begin{equation*}
    a_{\lambda_n}:=E(u_{\lambda_n};R_2),
\end{equation*}
\begin{equation*}
    b_{\lambda_n}:=E(u_{\lambda_n};R_1),
\end{equation*}
and $c_{\lambda_n}:=a_{\lambda_n}-b_{\lambda_n}$. 

We use the fundamental theorem of calculus and the monotonicity formula \eqref{derivative of E} to express $c_{\lambda_n}$ as 
\begin{align*}
    c_{\lambda_n}&=E(u_{\lambda_n};R_2)-E(u_{\lambda_n};R_1)=\int_{R_1}^{R_2}\frac{dE}{d\lambda}(u_{\lambda_n};\lambda)d\lambda=\int_{R_1}^{R_2}\frac{d}{d\lambda}(E(u_{\lambda_n\,\lambda};1))d\lambda\\
    &=\int_{R_1}^{R_2}\int_{\partial\,B_1\cap\Omega}\lambda\left({\partial_\lambda\,u_{\lambda_n\lambda}}{}\right)^2d\sigma\,d\lambda=\int_{R_1}^{R_2}\int_{\partial\,B_1\cap\Omega}\frac{1}{\lambda}\left(\frac{2}{p-1}u_{\lambda_n\lambda}+|x|{\partial_r\,u_{\lambda_n\lambda}}\right)^2d\sigma\,d\lambda\\
    &=\int_{R_1}^{R_2}\lambda^{\frac4{p-1}-d}\int_{\partial\,B_\lambda\cap\Omega}\left(\frac{2}{p-1}u_{\lambda_n}+|x|{\partial_r\,u_{\lambda_n}}\right)^2d\sigma\,d\lambda\\
    &=\int_{\Omega\cap B_{R_2}\setminus B_{R_1}}\vert x\vert^{\frac4{p-1}-d}\left(\frac{2}{p-1}u_{\lambda_n}+|x|{\partial_r\,u_{\lambda_n}}\right)^2\,dx.
\end{align*}
Since $(u_{\lambda_n})$ converges strongly in $H^1(\Omega_R)$, we deduce that
$$
\lim_{n\to+\infty} c_{\lambda_n} = \int_{\Omega\cap B_{R_2}\setminus B_{R_1}}\vert x\vert^{\frac4{p-1}-d}\left(\frac{2}{p-1}u_{\infty}+|x|{\partial_r\,u_{\infty}}\right)^2\,dx.
$$

Next, we  prove that \[c_{\lambda_n}\to0.\] By Lemma \ref{energy} and the trace inequality $\int_{\partial B_R \cap \Omega} u^2 d\sigma \leq C \big (R \int_{B_R \cap \Omega}\vert \nabla u \vert^2 dx + R^{-1} \int_{B_R \cap \Omega} u^2 dx \big) $, the sequences $(a_{\lambda_n})$ and 
$(b_{\lambda_n})$ are bounded. Using the monotonicity formula \eqref{derivative of E}, we deduce that $(a_{\lambda_n})$ and $(b_{\lambda_n})$ are nondecreasing and converge to the same finite limit.

Hence,
$$
0= \int_{\Omega\cap B_{R_2}\setminus B_{R_1}}
\vert x \vert^{\frac4{p-1}-d}\left(\frac{2}{p-1}u_{\infty}+|x|{\partial_r\,u_{\infty}}\right)^2\,dx= \int_{\Omega\cap B_{R_2}\setminus B_{R_1}}\vert x\vert^{2-d}\left( \partial_r (\vert x \vert^{\frac{2}{p-1}} u_{\infty})\right)^2\,dx
$$
so that $ w := \vert x \vert^{\frac{2}{p-1}} u_{\infty}$ is homogeneous of degree zero. Therefore we can find a measurable function $v\colon A\subset\mathbb{S}^{d-1}\to\R$ such that $w(x) = v \big(\frac{x}{\vert x \vert}\big)$ and thus $u_{\infty}$ has the form \eqref{sing solution}. Note that $w \in H^1\cap\,L^{p+1}(\Omega_2 \setminus \Omega_{\frac{1}{2}})$, since $u_{\infty} \in H^1\cap\,L^{p+1}(\Omega_R)$ for every $R>0$, and so $v\in\,H^1\cap\,L^{p+1}(A)$ by Fubini's theorem. To conclude we need to prove that $v$ belongs to $ H^1_0(A)$. To this end we set $w_n=\vert x \vert^{\frac{2}{p-1}} u_{\lambda_n}$ and  
observe that the sequence $(w_n)$ converges to $w$ in $H^1(\Omega_2 \setminus \Omega_{\frac{1}{2}})$, that is, $\int_{\frac{1}{2}}^2 \int_A \Big (\vert w_n - w \vert^2 + \vert \nabla (w_n - w) \vert^2\Big)(r\theta) r^{d-1} d\sigma dr \longrightarrow 0$, as $n \to \infty$.  
Therefore, up to a subsequence, for almost every $r \in (\frac{1}{2},2)$ we get that $\int_A \Big (\vert w_n - w \vert^2 + \vert \nabla (w_n - w) \vert^2\Big)(r\cdot) d\sigma  \longrightarrow 0$ and so we can find $ \bar r \in (\frac{1}{2},2)$ such that 
$$
\int_A \Big (\vert w_n - w \vert^2 + \vert \nabla' (w_n - w) \vert^2\Big)(\bar r \cdot) d\sigma  \longrightarrow 0,$$
where $\nabla'$ is the Riemannian gradient on the $(d-1)$-dimensional sphere. The latter and the fact that $w_{\vert A}=v $ imply that the sequence of $C^2(\overline{A})$ functions $ \theta \mapsto w_n(\bar r \theta)$ converges to $v$ in $H^1(A)$. It follows that $ v \in H^1_0(A)$ since $w_n(\bar r \cdot) = 0$ on $\partial A$.

\end{proof}



\subsubsection{Homogeneous and stable solutions}\label{homogeneous solutions}
In the previous section, we have shown that the blow-down limit $u_\infty$ of the rescaled solutions $u_\lambda$ is a weak solution of \eqref{main}, and it is homogeneous and stable. Here we study such solutions with the goal of proving that they are necessarily equal to $0$.

Given a homogeneous solution $u(r,\theta)=r^{-\frac{2}{p-1}}v(\theta)$ of \eqref{main}, one can make a simple formal computation to conclude that the nonradial factor $v$ satisfies the following equation
\begin{equation}\label{Eq on v}
    \begin{cases}
		-\Delta' v + \mu\,v \,=\, |v|^{p-1}v & \textrm{in}\;\; A\,, \\
		v \,=\, 0 & \textrm{on}\;\; \partial A\,,
	\end{cases}
\end{equation}
where $\mu:=\frac{2}{p-1}\left(d-1-\frac{p+1}{p-1}\right)$ and $-\Delta'$ is the Laplace-Beltrami operator on the $(d-1)$-dimensional sphere. If $u$ is also stable then $v$ satisfies the following inequality for suitable test functions $\psi$:
\begin{equation}\label{left of stable}
    \frac{(d-2)^2}{4}\int_{A}\psi^2d\sigma+\int_{A}|\nabla'\psi|^2d\sigma\geq\,p\int_{A}|v|^{p-1}\,\psi^2d\sigma,
\end{equation}
where $\nabla'$ is the Riemannian gradient on the $(d-1)$-dimensional sphere, see e.g. \cite[pp. 5245-5246]{Wang} for the proof\footnote{In \cite[pp. 5245-5246]{Wang}, $\psi$ is supposed to be smooth but the result remains valid for $\psi\in H^1_0(A)\cap L^{p+1}(A)$ with the same proof and can be extended to any $\psi\in H^1_0(A)$ by Fatou's lemma.} of \eqref{left of stable}.
We work in the class of weak solutions of \eqref{main} and make these considerations precise in the following lemma. 
\begin{lemma}\label{u to v}
Let $u$ be a weak, homogeneous, and stable solution of \eqref{main}. Then $v\in H^1_0\,\cap\,L^{p+1}(A)$ is a weak solution of \eqref{Eq on v} which satisfies the stability-type estimate \eqref{left of stable} for any $\psi\in\,H^1_0(A)$.
\end{lemma}

\begin{proof}
To prove \eqref{Eq on v}, we begin by testing \eqref{main} with the test function $\varphi(r,\theta)=\chi(r)\psi(\theta)$
where $\chi\colon\mathbb{R}_+\to[0,1]$ is a smooth, positive function compactly supported away from $0$,  and $\psi\in\,H^1_0(A)\cap L^{p+1}(A)$. Setting $\alpha=\frac{2}{p-1}$,
\begin{align*}
    \int_\Omega\nabla u\nabla\varphi\,dx&=\underbrace{\int_\Omega\,\partial_r(r^{-\alpha})v(\theta)\partial_r\chi(r)\psi(\theta)\,dx}_{I_1}+\underbrace{\int_\Omega\,r^{-\alpha-2}\chi(r)\nabla' v(\theta)\nabla'\psi(\theta)\,dx}_{I_2}\\
    &=\underbrace{\int_\Omega\,r^{-\frac{2p}{p-1}}\chi(r)|v(\theta)|^{p-1}v(\theta)\,\psi(\theta)\,dx}_{I_3}.
\end{align*}
We perform a separation of variables of the three integral terms.
\begin{align*}
    I_1&=-\alpha\int_0^\infty\,r^{-\alpha-2+d}\partial_r\chi(r)dr\int_{A}v(\theta)\psi(\theta)d\sigma\\
    &=\alpha(d-2-\alpha)\int_0^\infty\,r^{-\alpha-3+d}\chi(r)dr\int_{A}v(\theta)\psi(\theta)d\sigma.\\\\
    I_2&=\int_0^\infty\,r^{-\alpha-3+d}\chi(r)dr\int_{A}\nabla' v(\theta)\nabla'\psi(\theta)d\sigma,\\
    I_3&=\int_0^\infty\,r^{-p\alpha-1+d}\chi(r)dr\int_{A}|v(\theta)|^{p-1}v(\theta)\,\psi(\theta)d\sigma.
\end{align*}
Since $\alpha+2=p\alpha$, we can cancel all of the integral terms involving $r$ to conclude 
\begin{equation}\label{weak sol of v}
    \int_{A}\nabla' v\nabla'\psi\,d\sigma+\alpha(d-2-\alpha)\int_{A}v\psi\,d\sigma=\int_{A}|v|^{p-1}v\psi\,d\sigma.
\end{equation}
That is, $v\in H^1_0\cap L^{p+1}(A)$ is a weak solution of \eqref{Eq on v}. 
\end{proof}

We conclude this section by proving a Liouville-type result for weak solutions of \eqref{Eq on v} which satisfy estimate \eqref{left of stable}, under the additional condition for $p>1$
\begin{equation}\label{numerical condition}
    p\mu-\frac{(d-2)^2}{4}+(p-1)\lambda_1\geq0,
\end{equation}
where $\lambda_1$ is the first eigenvalue of the Laplace-Beltrami operator on $A$. Note that the term $(p-1)\lambda_1$ in \eqref{numerical condition} gives an improvement of the condition in \cite{Farina}: $1<p<p_{JL}(d)\iff p\mu-\frac{(d-2)^2}{4}>0 $. In other words, the condition $1<p<p_{JL}(d)$ is optimal only for stable solutions in the whole space $\R^d$.
\begin{lemma}\label{Liouville for v}
    Let $v\in H^1_0\cap L^{p+1}(A)$ be a weak solution of \eqref{Eq on v} which satisfies the stability-type estimate \eqref{left of stable}. Then for all $p$ satisfying \eqref{numerical condition}, $v=0$.
\end{lemma}

\begin{proof}
First, we take the solution $v$ (multiplied by $p$) as a test function in \eqref{weak sol of v} to obtain
\begin{equation*}\label{test with v2}
    p\mu\int_{A}v^2d\sigma+p\int_{A}|\nabla'v|^2d\sigma=p\int_{A}|v|^{p}v\,d\sigma.
\end{equation*}
Likewise, testing \eqref{left of stable} with $v$ yields
\begin{equation*}\label{test with v}
    \frac{(d-2)^2}{4}\int_{A}v^2d\sigma+\int_{A}|\nabla'v|^2d\sigma\geq\,p\int_{A}|v|^{p}v\,d\sigma.
\end{equation*}
Combining these two lines gives
\begin{equation*}
    \left(p\mu-\frac{(d-2)^2}{4}\right)\int_{A}v^2d\sigma+(p-1)\int_{A}|\nabla'v|^2d\sigma\leq0.
\end{equation*}
We use the trivial identity
\[\int_A|\nabla'v|^2d\sigma=-\left(\lambda_1\int_Av^2d\sigma-\int_A|\nabla'v|^2d\sigma\right)+\lambda_1\int_Av^2d\sigma\]
to deduce the following:
\begin{equation}\label{e3}
    \left(p\mu-\frac{(d-2)^2}{4}+\lambda_1(p-1)\right)\int_{A}v^2d\sigma\leq(p-1)\int_A\left(\lambda_1v^2-|\nabla'v|^2\right)d\sigma\leq0,
\end{equation}
where the final inequality follows from Poincar\'{e}'s inequality on $A$ with optimal constant $\lambda_1$.
By \eqref{numerical condition} and \eqref{e3}, $v$ is a principle eigenfunction of the Laplace-Beltrami operator on $A$. Combining this with the fact that $v$ also satisfies \eqref{Eq on v} leads trivially to the conclusion that $v=0$.
\end{proof}

\subsubsection{A Poho\v{z}aev-type result on subdomains of $\Sph$}
The following Poho\v{z}aev-type result is proven by Bidaut-V\'{e}ron, Ponce, and V\'{e}ron in \cite[Theorem 2.1]{BVPV} for smooth solutions $v$ to \eqref{Eq on v}. Since we deal with weak solutions, we construct smooth approximations $v_\lambda$ of $v$ and use the $H^1_{loc}$ convergence of $u_\lambda$ to $u_\infty$ (along a sequence) to prove that the inequality also holds for the blow-down limit $u_\infty$.
\begin{lemma}\label{pohozaev statement}
Let $A$ be a $C^{2,\alpha}$ domain of $\mathbb S^{d-1}_+$ which is star-shaped with respect to the north pole, and let $u_\infty(r,\theta)=r^{-\frac{2}{p-1}}v(\theta)$ be the blow-down limit of $(u_\lambda)$ as above. Then $v\in\,H^1_0(A)$ satisfies
\begin{equation}\label{poho half-sphere form}
    \left(\frac{d-3}{2}-\frac{d-1}{p+1}\right)\int_{A}|\nabla' v|^2\phi\,d\sigma-\frac{d-1}{2}\left(\frac{d-\mu\,(p-1)-1}{p+1}\right)\int_{A}\,v^2\phi\,d\sigma\leq0,
\end{equation}
where $\nabla'$ is the tangential gradient to $\Sph$, and $\phi(\theta)=\theta_d$ is an eigenfunction of the Laplace-Beltrami operator $-\Delta'$ in $H_0^1(\mathbb{S}^{d-1}_+)$ associated to the principal eigenvalue $\lambda_1(\mathbb{S}^{d-1}_+)=d-1$.
\end{lemma}

\begin{proof}
We begin with a simple geometric lemma. Denote by $\left<,\right>$ the inner product on the tangent space to $\mathbb{S}^{d-1}$, $\phi(\theta)=\theta_d$, and $\nabla'\phi$ its Riemannian gradient on $\mathbb{S}^{d-1}$. Then,
\begin{lemma}
Let $A\subset\mathbb S^{d-1}_+$ denote a $C^1$ open set with outward unit normal direction $\nu$. If $A$ is star-shaped with respect to the north pole, there holds $\left<\nabla'\phi,\nu\right>\le0$ on $\partial A$. 
\end{lemma}
\begin{proof}
Apply the stereographic projection $\pi_S$ from the south pole to the set $A$. Then take a point $x\in\pi_S(A)\subset B_1\subset\R^{d-1}$ and consider the unique minimal geodesic $\gamma_\theta$ contained in $A$ connecting the north pole to $\pi_S^{-1}(x)=:\theta$. Its projection $\pi_S(\gamma_\theta)$ is a straight line connecting the origin to $x$ and is contained in $\pi_S(A)$. So $\pi_S(A)\subset\R^{d-1}$ is star-shaped (in the usual Euclidean sense). Since the stereographic projection is conformal (it preserves angles), $\left<\nabla'\phi,\nu\right>$ has the same sign as $V\cdot n$, where $V=d\pi_S(\nabla'\phi)=\nabla(\frac{1-r^2}{1+r^2})=\frac{-4r}{(1+r^2)^2}\frac xr$, $x\in\R^{d-1}$, $r=\vert x\vert$, and $n=d\pi_S(\nu)$ is the outward unit normal direction of $\pi_S(A)$. In view of the Lemma in \cite{evans} p. 554, we conclude that $\left<\nabla'\phi,\nu\right>\le0$ on $\partial A$. 
\end{proof}
Let us return to the proof of Lemma \ref{pohozaev statement}. To simplify the appearance we set $C_1=\left(\frac{d-3}{2}-\frac{d-1}{p+1}\right)$ and $C_2=\frac{d-1}{2}\left(\frac{d-\mu\,(p-1)-1}{p+1}\right)$. Let $u\in\,C^2(\Omega)\cap\,C(\super\Omega)$ be a solution of \eqref{main}. Since for any $\lambda>0$, $u_\lambda(x)=\lambda^{\frac{2}{p-1}}u(\lambda\,x)$ is also a smooth solution of \eqref{main}, we set $v_\lambda$ to satisfy
\begin{equation*}
    u_\lambda(x)=r^{-\frac{2}{p-1}}v_\lambda(r,\theta).
\end{equation*}
Plugging this expression of $u_\lambda$ into \eqref{main}, it follows that  
$v_\lambda$ is a classical solution of the following equation:
\begin{equation}\label{eq on vlambda}  
    \begin{cases}
 		-\Delta' v_\lambda + \mu\,v_\lambda \,=\, |v_\lambda|^{p-1}v_\lambda + e_\lambda& \textrm{in}\;\; \Omega\,, \\
		v_\lambda \,=\, 0 & \textrm{on}\;\; \partial\Omega\,,
	\end{cases}
\end{equation}
where the error term $e_\lambda$ is
\begin{equation*}
    e_\lambda=\left(d-1-\frac{4}{p-1}\right)r\partial_rv_\lambda + r^2\partial_r^2v_\lambda.
\end{equation*}
Again for simplicity, we define $C_3=\left(d-1-\frac{4}{p-1}\right)$.

Following \cite[pp. 186-188]{BVPV}, we apply the divergence theorem to the vector field
\begin{equation*}
    P=\left<\nabla'\phi,\nabla' v_\lambda\right>\nabla' v_\lambda,
\end{equation*}
to eventually reach the following relation
\begin{equation}\label{div}
    \begin{aligned}
        C_1\int_A|\nabla'v_\lambda|^2\phi\,d\sigma+C_2\int_Av_\lambda^2\phi\,d\sigma-\frac{d-1}{p+1}\int_Ae_\lambda v_\lambda\phi\,d\sigma+\int_A\left<\nabla'\phi,\nabla'v_\lambda\right>e_\lambda d\sigma\\
        =\frac{1}{2}\int_{\partial A}|\nabla'v_\lambda|^2\left<\nabla'\phi,\nu\right>d\tau\leq0,
    \end{aligned}
\end{equation}
where $\nu$ is the unit outer normal. The right-hand side of \eqref{div} is nonpositive since $A$ is star-shaped. Next, we multiply all terms of \eqref{div} by a standard cutoff function $\eta\colon\R_+\to\R_+$ compactly supported away from $0$ and we integrate in the radial direction (including the factor $r^{d-1}$). We begin with the first error term:
\begin{align*}
    \int_0^\infty\int_A&e_\lambda v_\lambda\phi r^{d-1}\eta\,d\sigma\,dr=C_3\int_0^\infty\int_A(\partial_rv_\lambda) v_\lambda\phi r^{d}\eta\,d\sigma\,dr+\int_0^\infty\int_A(\partial_r^2v_\lambda) v_\lambda\phi r^{d+1}\eta\,d\sigma\,dr\\
    &=C_3\int_0^\infty\int_A(\partial_rv_\lambda) v_\lambda\phi r^{d}\eta\,d\sigma\,dr-\int_0^\infty\int_A(\partial_rv_\lambda)^2 \phi r^{d+1}\eta\,d\sigma\,dr-\int_0^\infty\int_A(\partial_rv_\lambda)v_\lambda \phi \partial_r(r^{d+1}\eta)\,d\sigma\,dr,
\end{align*}
where we have integrated by parts in the second integral of the first line. Now, by the fact that $v_\lambda\to{v}=r^{\frac{2}{p-1}} u_\infty$ in $H^1(\Omega_R)$ (along a sequence) and that ${v}$ does not depend on $r$, we can pass to the limit as $\lambda\to\infty$ to conclude
\[\lim_{\lambda\to\infty}\int_0^\infty\int_Ae_\lambda v_\lambda\phi r^{d-1}\eta\,d\sigma\,dr=0.\]
For the second error term in \eqref{div}, we proceed as follows:
\begin{align*}
    \int_0^\infty\int_A\left<\nabla'v_\lambda,\nabla'\phi\right>e_\lambda r^{d-1}\eta\,d\sigma dr&=\int_0^\infty\int_A\left<\nabla'v_\lambda,\nabla'\phi\right>\left(C_3r\partial_rv_\lambda+r^2\partial_r^2v_\lambda\right) r^{d-1}\eta\,d\sigma dr\\
    &=\underbrace{C_3\int_0^\infty\int_A\left<\nabla'v_\lambda,\nabla'\phi\right>\partial_rv_\lambda\,r^{d}\eta d\sigma dr}_{I_1}\\&\underbrace{-\int_0^\infty\int_A\left<\nabla'v_\lambda,\nabla'\phi\right>\partial_rv_\lambda\, \partial_r(r^{d+1}\eta) d\sigma dr}_{I_2}\\
    &\underbrace{-\int_0^\infty\int_A\left<\nabla'\partial_rv_\lambda,\nabla'\phi\right>\partial_rv_\lambda r^{d-1}\eta d\sigma dr}_{I_3},
\end{align*}
where we have integrated by parts in the $r$ variable. Since $v_\lambda\to{v}$ in $H^1(\Omega_R)$ (along a sequence) and since ${v}$ does not depend on $r$, $I_1$ and $I_2$ tend to zero as $\lambda\to\infty$. We deal with $I_3$ by recalling that $\phi$ is an eigenfunction and so
\begin{align*}
-\int_0^\infty\int_A\left<\nabla'\partial_rv_\lambda,\nabla'\phi\right>\partial_rv_\lambda r^{d-1}\eta d\sigma dr&=-\frac{1}{2}\int_0^\infty\int_A\left<\nabla'(\partial_rv_\lambda)^2,\nabla'\phi\right> r^{d-1}\eta d\sigma dr\\
&=
-\frac{d-1}{2}\int_\Omega(\partial_rv_\lambda)^2\phi \eta\,dx,
\end{align*}
where in the latter we have used an integration by parts over $A$ and the fact that $\partial_rv_\lambda =0$ on $\partial \Omega \setminus \{0 \}$. 

This term then tends to $0$ as $\lambda\to\infty$, since $v_\lambda$ converges to ${v}$ in $H^1(\Omega_R)$ (along a sequence).

Now we deal with the other solid integral terms arising from integrating \eqref{div} with the test function $r^{d-1}\eta$ in the $r$ variable. Since $u_\lambda$ converges to $u$ in $H^1(\Omega_R)$ for every $R>0$ along a sequence, so does $v_\lambda$ to ${v}$. Since ${v}$ does not depend on $r$, we can apply Fubini's theorem at the limit to obtain
\begin{equation*}
    \begin{aligned}
        \lim_{\lambda\to\infty}\int_0^\infty\int_A \left(C_1|\nabla'v_\lambda|^2+C_2v_\lambda^2\right)\phi\,r^{d-1}\eta d\sigma dr
        =\int_0^\infty r^{d-1}\eta dr\int_A\left(C_1|\nabla'v|^2+C_2v^2\right)\phi d\sigma\leq0,
    \end{aligned}
\end{equation*}
where $C_1$ and $C_2$ are the fixed constants defined at the beginning of the proof. Dividing by the positive constant $\int_0^\infty r^{d-1}\eta dr$ yields \eqref{poho half-sphere form}.
\end{proof}

\subsubsection{Proof of Theorem \ref{To prove:SOAC2}}\label{eop}
\begin{proof}
By an elementary computation, $p_S(d)<p_S(d-1)<p_{JL}(d)$. Recall also that if $p\le p_{JL}(d)$, then \eqref{numerical condition} holds. So all $p>p_S(d)$ either satisfy \eqref{numerical condition}, $p>p_S(d-1)$, or both. Suppose first that \eqref{numerical condition} holds. By Lemmas \ref{properties of uinf}, \ref{u to v}, and \ref{Liouville for v}, it follows that $u_\infty=0$. In the remaining case $p>p_S(d-1)$, Lemma \ref{pohozaev statement} implies that $v=0$ (which of course implies that $u_\infty=0$). Then, since $(u_{\lambda_n})$ converges strongly to $u_\infty$ in $H^1\cap\,L^{p+1}(\Omega_R)$, it follows that $E(u_{\lambda_n};1)\to0$ as $\lambda_n\to\infty$. Furthermore, by the monotonicity formula \eqref{derivative of E}, the convergence to $0$ is monotone non-decreasing. 

To conclude, we study the behavior of $E(u_{\lambda};1)=E(u;\lambda)$ as $\lambda\to0^+$. By definition \eqref{E} of the monotonicity formula, we have
$$
E(u,\lambda) \ge -\frac{1}{p+1} 
\lambda^{2\frac{p+1}{p-1}-d}\int_{\Omega_\lambda} \vert  u\vert^{p+1}\;dx = -c\lambda^{2\frac{p+1}{p-1}}\fint_{\Omega_\lambda} \vert  u\vert^{p+1}\;dx
$$
Since $u$ is bounded in a neigbourhood of $0$, we deduce that $\lim_{\lambda\to 0^+} E(u;\lambda)\ge 0$. Since $E(u,\cdot)$ is nondecreasing, we deduce that
\begin{equation}\label{eispos}
    E(u;\lambda)\ge 0\quad\text{for all $\lambda>0$}
\end{equation}
Now, $\lim_{\lambda\to+\infty} E(u;\lambda)=E(u_\infty;1)=0$ and so, since $E(u,\cdot)$ is nondecreasing and nonnegative, $E(u;\lambda)=0$ for every $\lambda>0$. By \eqref{derivative of E}, $u$ is homogeneous, which is possible only if $u=0$ since $u\in C(\super\Omega)$. 
\end{proof}

%
%
%

\section{Proofs of Corollary \ref{cor:SOAC} and Corollary \ref{lastcor} in the supercritical case}

\subsection{Proof of Corollary \ref{cor:SOAC}} We just need to inspect the proof of Theorem \ref{To prove:SOAC} and adapt it as follows. Let $u$ be a nonnegative weak stable solution.
By Remark \ref{form-monotonia-coni-troncati}, the monotonicity formula remains valid and so\footnote{In the second part of the proof of Lemma \ref{properties of uinf}, in order to prove that $v\in H^1_0(A)$, one just needs to replace the sequence $(u_{\lambda_n})$ by the approximating sequence $(u_n^{R})$, $R=2$, from Remark \ref{form-monotonia-coni-troncati}.} does Lemma \ref{properties of uinf}. So, the blow-down limit $u_\infty$ is homogeneous.
Regarding Pohozaev's identity Lemma \ref{pohozaev statement}, we use Proposition 13 in \cite{df} to deduce that $u_\infty$ is also the limit (in $H^1(B_R^+)$) of classical solutions $u_n^R$ (with zero value on $\partial\R^d_+\cap B_R$) and so the same proof applies. Finally, we adapt Section \ref{eop} by observing that by the same proof $E(u_n^R;\lambda)\ge 0$ for all $\lambda\in (0,R)$ and so $E(u;\lambda)$ is nonnegative for all $\lambda>0$. As before, this forces $u$ to be homogeneous, hence $u=0$ by our classification of homogeneous weak solutions. 

\hfill\qed

\subsection{Proof of Corollary \ref{lastcor}}
In the supercritical case $p>p_S(d)$, the proof is identical to that of Theorem \ref{To prove:SOAC2}. We only need to observe that in the range \eqref{numerical condition} of exponents $p$, no geometrical condition on the cone $\Omega$ is needed.
\hfill\qed

\subsection{Positive solutions in convex cones}
In this last section, we explain more carefully how Busca's result \cite{busca} for positive classical solutions defined on convex cones can be recovered thanks to Theorem \ref{To prove:SOAC2}. This is the content of Lemma \ref{geometric2} below. 
First, we extend the following natural result proven in \cite[Proposition 2]{Ferreira} in the case where $\Omega$ is convex. 
\begin{lemma}\label{geometric0}
Let $\Omega$ be defined as in \eqref{cone} for some $A\subset\mathbb{S}^{d-1}_+$. Then $\Omega$ is convex (resp. star-shaped) if and only if $A$ is geodesically convex (resp. geodesically star-shaped).
\end{lemma}
\begin{proof}
Let $\Omega$ be convex. Now take two points $\varphi_1,\varphi_2\in A$. (By the hypothesis that $A\subset\mathbb{S}^{d-1}_+$, $\varphi_1$ and $\varphi_2$ may not be antipodal points.) By convexity, $t\varphi_2+(1-t)\varphi_1\in\Omega$, so \[S_t=\frac{t\varphi_2+(1-t)\varphi_1}{|t\varphi_2+(1-t)\varphi_1|}\in A\]for all $t\in[0,1]$ by the cone property of $\Omega$. It can be directly verified, for example using a stereographic projection from $-\varphi_2$, that $\{S_t\}_{0\leq t\leq1}$ coincides with the unique geodesic curve from $\varphi_1$ to $\varphi_2$. Finally, since $\{S_t\}_{t>0}\subset A$, we have that $A$ is geodesically convex.

Conversely, let us choose two general points $x_1,x_2\in\Omega$. In the same way as above, we can show that 
\begin{equation*}
\left\{\frac{tx_2+(1-t)x_1}{|tx_2+(1-t)x_1|}\right\}_{0\leq t\leq 1}\subset A
\end{equation*}
is a geodesic curve on $A$, whereby we find that $tx_2+(1-t)x_1\in\Omega$ due to the cone property of $\Omega$.

In the case of star-shaped domains, we repeat the above method with $\varphi_2$ and $x_2$ fixed taken to be the points with respect to which $A$ and $\Omega$, respectively, are star-shaped.

\end{proof}

\begin{lemma}\label{geometric1}
If $A\subset\mathbb{S}^{d-1}_+$ is star-shaped with respect to some direction $\varphi_0\in A$, then $\Omega$ is convex in that direction in the sense that for all $x\in\Omega$ and $t>0$, $x+t\varphi_0\in\Omega$. In particular, the boundary of the cone $\Omega$ is a graph 
with respect to the direction $\varphi_0\in A$.
\end{lemma}

\begin{proof}
First, by virtue of the radial scaling property of cones, it is enough to prove the result for all $x=\varphi_1\in A$. It can then be verified that the curve \[S=\{S_t\}_{t>0}=\left\{\frac{\varphi_1+t\varphi_0}{|\varphi_1+t\varphi_0|}\right\}_{t>0}\subset\mathbb{S}^{d-1}_+,\]
coincides with the unique minimal geodesic connecting $\varphi_1$ to $\varphi_0$ (for example by taking the stereographic projection from $-\varphi_0$). Then, since $A$ is star-shaped with respect to $\varphi_0$, $S_t\in A$ for all $t$, whereby $x+t\varphi_0\in\Omega$ due to the cone property of $\Omega$.
\end{proof}

\begin{lemma}\label{geometric2}
Assume that $A\subset\mathbb{S}^{d-1}_+$ is geodesically convex. Then, a suitable rotation of $A$ is star-shaped with respect to the north pole and still contained in $\mathbb{S}^{d-1}_+$.
\end{lemma}
\begin{proof}
If the north pole $\Vec{n}$ belongs to $A$ the statement is trivially proven because convex sets are in particular star-shaped. Therefore, let us assume that $\Vec{n}$ does not belong to $A$. Then $\Omega\cap\left<\Vec{n}\right>=\emptyset$ by the cone property of $\Omega$, and $\Omega$ is convex by Lemma \ref{geometric0}. So by the geometric form of the Hahn-Banach theorem, there exists a hyperplane $H$ normal to some $\nu\in\mathbb{S}^{d-1}$ (w.l.o.g. take $\nu=(-1,0,\dots,0)$) which includes the $x_d$ axis and which satisfies $H\cap\Omega=\emptyset$. By the convexity assumption, $\Omega$ (resp. $A$) must lie on one side of $H$ (resp. $H\cap\mathbb{S}^{d-1}_+$). That is, w.l.o.g, $A\subset\mathbb{S}^{d-1}_+\cap\{x_1>0\}$.

Let us consider the great circle $G=\{(x_1,0,\dots,0,x_d)\colon x_i\in[-1,1],|x|=1\}$.
Up to a suitable modification of the choice of $\nu$, we can assume (by convexity of $A$) that there is a point $a=(\cos\theta_a,0,\dots,0,\sin\theta_a)\in G\cap A$ for its corresponding $\theta_a\in(0,\pi/2)$. Then we take the rotation that sends $a$ to $\Vec{n}$. This rotation has a specific form: for a general $p=(\alpha_p\cos\theta,x_2,\cdots,x_{d-1},\alpha_p\sin\theta)\in A$ for its corresponding $\alpha_p\in[0,1]$, $\theta\in(0,\pi/2)$, we have
\begin{equation*}
    p\mapsto\left(\alpha_p\cos\left(\frac{\pi}{2}-\theta_a+\theta\right),x_2,\cdots,x_{d-1},\alpha_p\sin\left(\frac{\pi}{2}-\theta_a+\theta\right)\right).
\end{equation*}
Since $\theta,\theta_a\in(0,\pi/2)$, it can be verified that the image of $A$ under this rotation is still contained in $\mathbb{S}^{d-1}_+$.

\end{proof}

\appendix
\section*{Acknowledgement}
This work was performed within the framework of the LABEX MILYON (ANR-10-
LABX-0070) of Université de Lyon, within the program "Investissements
d'Avenir" (ANR-11-IDEX-0007) operated by the French National Research Agency
(ANR).

\bibliographystyle{plain} 

\end{document}